\documentclass[11pt]{amsart}
\usepackage[margin=1.0in]{geometry}
\usepackage{amsmath,amsthm,amssymb}
\usepackage{esint}
\usepackage{tikz}
\usepackage{bm}
\usepackage{color}
\definecolor{grn}{rgb}{0,0.4,0}
\definecolor{dgrn}{rgb}{0.0,0.3,0.0}
\definecolor{dpur}{rgb}{0.3,0.0,0.6}
\usepackage[mathscr]{euscript}
\usepackage[colorlinks=true, linkcolor=dpur]{hyperref}

\numberwithin{equation}{section}
\newtheorem{prop}{Proposition}
\newtheorem{lemma}[prop]{Lemma}

\newtheorem{cor}[prop]{Corollary}

\numberwithin{prop}{section}
\theoremstyle{definition}

\newtheorem{rmk}[prop]{Remark}

\newtheorem*{thm*}{Theorem}
\newtheorem*{oprob}{Open Problem}

\author[M.~Petrache]{Mircea~Petrache}
\address{Pontificia Universidad Catolica de Chile, Santiago, Chile} \email{decostruttivismo@gmail.com}

\thanks{The author acknowledges support from the FONDECYT \emph{Iniciacion en Investigacion 2017} grant N. 11170264.}
\begin{document}
\title[Cyclical Monotonicity]{Cyclically monotone non-optimal $N$-marginal transport plans and Smirnov-type decompositions for $N$-flows.}

\date{\today}
\begin{abstract}
In the setting of optimal transport with $N\ge 2$ marginals, a necessary condition for transport plans to be optimal is that they are $c$-cyclically monotone. For $N=2$ there exist several proofs that in very general settings $c$-cyclical monotoncity is also sufficient for optimality, while for $N\ge 3$ this is only known under strong conditions on $c$. Here we give a counterexample which shows that $c$-cylclical monotonicity is in general not sufficient for optimality if $N\ge 3$. Comparison with the $N=2$ case shows how the main proof strategies valid for the case $N=2$ might fail for $N\ge 3$. We leave open the question of what is the optimal condition on $c$ under which $c$-cyclical monotonicity is sufficient for optimality. The new concept of an $N$-flow seems to be helpful for understanding the counterexample: our construction is based on the absence of finite-support $N$-cycles in the set where our counterexample cost $c$ is finite. To follow this idea we formulate a Smirnov-type decomposition for $N$-flows.
\end{abstract}
\maketitle 
\tableofcontents

\section{Introduction, basic definitions and setting}\label{sec:combred}

Let $c:\vec X:=X_1\times X_2\times\ldots\times X_N\to Y$ be a cost, where $N\ge 2$, $X_1,\ldots,X_N$ are Polish spaces, and $Y=(Y,+,\ge)$ is an ordered group (usually we take $Y=\mathbb R$ or $Y=(-\infty,+\infty]$). If $\mu_1,\ldots,\mu_N$ are probability measures with $\rho_j\in\mathcal P(X_k)$ and $\pi_j$ is the canonical projection of $\vec X$ onto $X_k$, then the possible ways to couple the $N$ measures $\mu_k, 1\le k\le N$ are encoded by the space of so-called \emph{transport plans} defined as 
\[
\Pi(\mu_1,\ldots,\mu_N):=\left\{\gamma\in\mathcal P(\vec X):\ (\pi_k)_\#\gamma = \mu_k\ \mbox{ for }1\le k\le N\right\}.
\]
The \emph{$N$-marginal optimal transport problem} with cost $c$ as above with \emph{marginals $\mu_k,1\le k\le N$}, is the following minimization problem:
\begin{equation}\label{ot}
 \min\left\{\langle c, \gamma\rangle:\ \gamma\in\Pi(\mu_1,\ldots,\mu_N)\right\},\quad\mbox{where}\quad \langle c,\gamma\rangle:=\int_{\vec X}c(\vec x)d\gamma(\vec x).
\end{equation}
Existence of minimizers can be ensured e.g. by assuming that $c$ is lower semicontinuous and bounded below, by the direct method of the calculus of variations and via Prokhorov's theorem. We do not focus on existence and uniqueness in the present work, and for our purposes, when needed, we can just assume that minimizers exist. Both in the theoretical development and in algorithmic implementations, it becomes important to limit the class of plans $\gamma$ that we are required to test in the minimization above (see \cite{frieseckevogler} and the references therein). 

\medskip

One important property which in many cases ensures optimality of $\gamma$, and which is the main focus of the present paper, is $c$-monotonicity. This condition was introduced in the case of $N=2$ marginals as a natural extension valid in the case $c(x,y)=\langle x,y\rangle$ for $x,y\in \mathbb R^d$, and makes a natural connection to convex analysis (see \cite{Villani} and \cite{Rockafellar}). For a recent treatment of the case of general $N$ from the point of view of the analogy with convex geometry, see \cite{bbbrw}, \cite{bbw} and the references therein. We now pass to give the precise definition for general $N$.

\medskip

A set $A\subset \vec X$ is $c$-monotone (or $c$-cyclically monotone) if for all $\ell\in \mathbb N$ and every $N$-ple of permutations $\vec \sigma=(\sigma_1,\ldots,\sigma_N)\in (S_\ell)^N$, for every $\ell$-ple $\vec x=(\vec x_1,\ldots,\vec x_\ell)\in (\vec X)^\ell$ with $\vec x_j=(x_j^1,\ldots,x_j^N)\in A$ for each index $j=1,\ldots,\ell$ there holds
\begin{equation}\label{cycmon}
 \sum_{j=1}^\ell c(\vec x_j)\le \sum_{j=1}^\ell c((\vec\sigma\cdot\vec x)_j),\quad\mbox{where}\quad (\vec\sigma\cdot\vec x)_j:=(x_{\sigma_1(j)}^1,x_{\sigma_2(j)}^2,\ldots,x_{\sigma_N(j)}^N)\in \vec X.
\end{equation}
As pointed out for example in \cite{Griessler}, this is equivalent to requiring that for any finitely supported measure $\alpha\in\mathcal P(\mathbb N^3)$ which is absolutely continuous with respect to $\gamma$, the plan $\alpha$ has minimal cost amongst plans with the same marginals as $\alpha$. Indeed, assume we have a probability measure $\gamma_N\in \mathcal P(\vec X)$ with marginals $\pi_k(\gamma_N)=\mu_k\in \mathcal P(X_k), k=1,\ldots,N$ and such that $\gamma_N$ is an atomic measure with finitely many atoms, i.e. 
\[
 \gamma_N=\frac1\ell\sum_{j=1}^\ell \delta_{\vec x_j} \quad\mbox{and thus}\quad \pi_k(\gamma_N)=\frac1\ell\sum_{j=1}^\ell\delta_{x_j^k}\quad\mbox{and}\quad \langle c,\gamma_N\rangle =\frac1\ell \sum_{j=1}^\ell c(\vec x_j).
\]
Then we can replace $\gamma_N$ by the plan $\tilde\gamma_N:=\vec\sigma(\gamma_N)$ where the action of $(S_\ell)^N$ on atomic probability measures with $\ell$ equal atoms $\mathcal P_\ell(\vec X)$ is inherited from the action on $(\vec X)^\ell$ by composing with the surjective ``empirical measure'' map
\begin{equation}
 (\vec X)^\ell\ni(\vec x_1,\ldots,\vec x_\ell)\mapsto \frac1\ell \sum_{j=1}^\ell \delta_{\vec x_j}\in \mathcal P_\ell(\vec X).
\end{equation}
For $\tilde\gamma_N$, explicitly we have 
\[
\tilde\gamma_N=\frac1\ell\sum_{j=1}^\ell \delta_{(\sigma\cdot\vec x)_j} \quad\mbox{and thus}\quad \pi_k(\tilde\gamma_N)=\frac1\ell\sum_{j=1}^\ell \delta_{x_{\sigma_k(j)}^k}= \pi_k(\gamma_N)\quad\mbox{and}\quad\langle c,\gamma_N\rangle =\frac1\ell \sum_{j=1}^\ell c((\sigma\cdot\vec x)_j).
\]
As is well known (see \cite{frieseckevogler} and the references therein), the above measures are the extremals of the set of probabilities with marginals equal to those of $\gamma_N$, and thus the linear optimization problem \eqref{ot} achieves its minimum on this set. Therefore \eqref{cycmon} implies that $\gamma_N$ has cost lower or equal than any competitor $\tilde \gamma_N=\vec\sigma(\gamma_N)$ in this case.

\medskip

In the case $N=2$ the link between $c$-monotonicity of the support of transport plans and their optimality is well understood, see \cite{pratelli}, \cite{beiglbock}, \cite{teichmair} and \cite{bianchinicaravenna} and the references therein. Recent progress has been made in \cite{Griessler}, \cite{bbbrw} and \cite{depascale} for the case of general $N$, with special focus on the case of costs $c$ coming from potential theory and mathematical physics.

\medskip

However a result at the same level of generality as e.g. \cite{pratelli} or \cite{beiglbock} is missing for $N\ge 3$. The result \cite{pratelli} proves that $c$-monotonicity is equivalent to optimality for the case of general $c$ in the case of atomic marginals, with the only requirement being that there exists a finite-cost plan. A consequence of the counterexample from the present work is that such general result is simply false for any $N\ge 3$, due to phenomena absent in the case $N=2$. This makes the characterization of optimality even more interesting, and a possible source of new mathematics, for $N\ge3$. Note that $\#(S_\ell)^N=(\ell!)^N$ is of exponential growth in both $N$ and $\ell$, thus even in the case of measures with finitely many atoms, condition \eqref{cycmon} is prohibitively hard to test in practice. It is even prohibitive to efficiently store on a computer all the competitors appearing in \eqref{cycmon}. The later problem is addressed in \cite{frieseckevogler}, to which we refer for further references in this direction.

\section{The counterexample}

It has been proved in \cite{pratelli} that for $N=2$ marginals, $c$-cyclical monotonicity is equivalent to optimality of transport plans. The same has been proved under stronger conditions for a general number of marginals, but the question remained open, of whether or not for $N>2$ marginals the no-hypothesis theorem from \cite{pratelli} about the case of atomic measures also holds or if the ergodic approach \cite{beiglbock} can be adapted. On a space $X$, we find:
\begin{equation}\label{counterexample}
 \left\{\begin{array}{l}\mbox{a cost }c:X^N\to (0,+\infty],\\ \mbox{a measure }\mu\in\mathcal P(X)\\ \mbox{a finite-cost plan }\gamma\in\mathcal P_\mathsf{sym}(X^N)\end{array}\right. \mbox{such that }\left\{\begin{array}{l}\gamma\mbox{ is }c\mbox{-cyclically monotone},\\\mbox{the marginal of }\gamma\mbox{ is }\mu\\\gamma\mbox{ is not $c$-optimal.}\end{array}\right.
\end{equation}
\subsection{A simple construction for $c$ taking the value $+\infty$}
\begin{prop}\label{prop:counterexample1}
There exist $(c,\mu,\gamma)$ satisfying \eqref{counterexample} with $X=\mathbb N$ and $N=3$.
\end{prop}

\begin{rmk}
\begin{itemize}
\item \textit{(Other $N\ge 3$)} The case of more than three marginals follows from the above, e.g. by defining a cost $c$ which only depends on three coordinates. 
\item \textit{(Same result on other metric spaces)} The counterexample construction which we will do works on the space $X=\mathbb N$, however $\mathbb N$ can be embedded injectively into another space $Y$ of at least countable cardinality, such as $Y=\mathbb R$, in several ways: any sequence $(a_k)_{k\in\mathbb N}$ such that $a_k\in Y$ are distinct is such an embedding. In our proof only the value of $c$ over the image $A_Y:=\{a_k:\, k\in\mathbb N\}$ of such embedding is used, and thus $c$ could be extended arbitrarily outside $A_Y$. Therefore one can generate a series of counterexamples in any such $Y$ as well. 
\item \textit{(The fact that $X$ is infinite is crucial)} For finite $X$, monotonicity implies optimality of finite-$c$-cost plans, for any $c$ and any $N\ge 2$. This follows by direct verification from the $\ell=\#X$ case of the monotonicity condition \eqref{cycmon}. 
\item \textit{(Growth requirements on the cost)} The proof becomes clearer if we first allow $c=+\infty$ on a large subset of $\mathbb N^3$. So we first provide the proof in this case, and then provide a second proof for finite $c$. 

We cannot have Proposition \ref{prop:counterexample1} to hold for $c$ too tame. Indeed Griessler \cite{Griessler} showed that for $N=3$, if $c(a,b,c)\le f(a)+f(b)+f(c)$ holds $\mu$-almost everywhere for some $f\in L^1(\mu)$, then $c$-cyclical monotonicity implies optimality. The positive result also holds also for general $N$ and can be extended by the De Pascale method \cite{depascaledual} to more singular $c$, see also \cite{depascale}.
\end{itemize}
\end{rmk}

\begin{proof}[{Proof of Proposition \ref{prop:counterexample1}:}]\hfill

\noindent\textbf{A cost taking the value $+\infty$.} The cost $c$ will be defined in terms of an auxiliary bounded function $f:\mathbb N\to(0,1]$ as follows:
\begin{itemize}
 \item $c(a,b,c)$ is symmetric with respect to permutations of the triple $a,b,c$.
 \item $c(1,1,1):=1$.
 \item $c(a,a,a+1):=f(a)$ for $a\ge 1$.
 \item $c(a,b,c):=+\infty$ for all the\ triples $\{a,b,c\}$ not described in the previous two cases.
\end{itemize}
\textbf{Required properties of the function $f$.} We don't need to fully specify the values of $f$ but we require that 
\begin{equation}\label{require_f}
\sum_{k=1}^\infty 4^{-k}(f(2k-1)-f(2k))>\frac16.
\end{equation}
(This is achieved, for example, if $f(1)>f(2)+\frac23$ and $f$ is decreasing.)

\medskip

\noindent\textbf{Two plans and their common marginal.} We define
\begin{eqnarray}
 \mu&:=&\sum_{k=1}^\infty 2^{-k}\delta_k,\\
 \gamma&:=&\sum_{k=1}^\infty4^{-k}\left(\delta_{2k-1,2k-1,2k}+\delta_{2k-1,2k,2k-1}+\delta_{2k,2k-1,2k-1}\right),\\
 \bar\gamma&:=&\frac12 \delta_{1,1,1} + \frac12\sum_{k=1}^\infty 4^{-k}\left(\delta_{2k,2k,2k+1}+\delta_{2k,2k+1,2k}+\delta_{2k+1,2k,2k}\right).
\end{eqnarray}
It is easy to check that $\gamma,\bar\gamma\in\mathcal P_\mathsf{sym}(\mathbb N^3)$ and that $\mu$ is the first marginal of $\gamma$ and of $\bar \gamma$. The two plans $\gamma$ and $\bar\gamma$ have equal marginals because the plans are symmetric. 

\medskip

\noindent\textbf{The plan $\gamma$ is non-optimal.} The $c$-cost of $\bar\gamma$ is lower than the $c$-cost of $\gamma$ due to the properties of $c,f$:
\begin{eqnarray*}
 \langle c,\gamma\rangle - \langle c,\bar\gamma\rangle&=&3\sum_{k=1}^\infty 4^{-k}f(2k-1) - \frac12 -3\sum_{k=1}^\infty4^{-k}f(2k)\\
 &=&3\sum_{k=1}^\infty 4^{-k}(f(2k-1)-f(2k))-\frac12\\
 &\stackrel{\mbox{\eqref{require_f}}}{>}&0.
\end{eqnarray*}
\textbf{The plan $\gamma$ is $c$-cyclically monotone.} To prove the above, we first note that, due to the symmetry of the problem, without loss of generality we may restrict to symmetric measures with atoms corresponding to integers smaller or equal to $M$:
\begin{equation}\label{defalpha}
 \alpha=\sum_{k=1}^M\alpha_k\left(\delta_{2k-1,2k-1,2k}\right)^\mathsf{sym}, \quad\mbox{with}\quad \alpha_k\ge 0, \quad \sum_{j=1}^M \alpha_k=1.
\end{equation}
The marginals of the above $\alpha$ are all equal to
\[
\mu_\alpha:=\sum_{k=1}^M\alpha_k\left(\frac23\,\delta_{2k-1} + \frac13\,\delta_{2k}\right).
\]
We need to show that if $\alpha'\in\mathcal P_\mathsf{sym}(\mathbb N^3)$ is another plan with the same marginals then $\langle c,\alpha'\rangle \ge \langle c, \alpha\rangle$. 

If $\langle c,\alpha'\rangle=+\infty$ then the desired inequality holds, so we are left with the case $\langle c, \alpha'\rangle<+\infty$. In this case $\alpha'$ can be written in the form
\begin{equation*}
 \alpha'=\bar a\,\delta_{1,1,1} + \sum_{k=1}^M\left[a_k\,\left(\delta_{2k-1,2k-1,2k}\right)^\mathsf{sym} + b_k\,\left(\delta_{2k,2k,2k+1}\right)^\mathsf{sym}\right],\quad
 \bar a,a_k,b_k\ge 0, \quad \bar a +\sum_{k=1}^M(a_k+b_k)=1,
\end{equation*}
and the marginal condition on $\alpha'$ translates into
\begin{eqnarray}\label{alpha1_bara}
 2\alpha_1&=& 3\bar a +2 a_1\nonumber\\
 2\alpha_k&=& 2a_k + b_{k-1},\quad \mbox{for}\quad k\ge 2,\nonumber\\
 \alpha_k&=& a_k +2 b_k,\quad \mbox{for}\quad k\ge 1.\label{alphak}
\end{eqnarray}
Note that the equations \eqref{alphak} imply that
\[
 a_k +2b_k=\alpha_k = a_k+\frac{b_{k-1}}{2},\quad\mbox{for}\quad k\ge 2,
\]
This means that $b_{k+1}= \frac14 b_k$ for $k\ge 1$. For $k> M$ we have $b_k=0$ thus we find $b_k=0$ for all $k\ge 1$, and by \eqref{alphak} we get $\alpha_k=a_k$ for all $k$, which inserted in \eqref{alpha1_bara} gives $\bar a =0$. Thus $\alpha'=\alpha$ is the only symmetric plan with marginals equal to $\alpha$ and finite cost. As this is true for every $\alpha\ll\gamma$ with finitely many atoms, we have that $\gamma$ is $c$-cyclically monotone, as desired.
\end{proof}

\subsection{About the existence of everywhere finite $c$}
In the search for a counterexample $c$ which is now everywhere finite, we first consider the problem abstractly, in a very general setting.

\medskip 

Assume that there exist $\gamma, \overline\gamma$ finite-cost transport plans, such that $\gamma$ is non-optimal and cyclically monotone, and $\overline\gamma$ is optimal. Then we consider the following sets of measures:
\begin{eqnarray*}
 \mathcal F_\gamma&:=&\left\{\alpha \in \mathcal P(\vec X),\ \#\left(\mathrm{spt}(\alpha)\right)<\infty, \ \alpha\ll \gamma\right\},\\
 \overline{\mathcal F}_\gamma&:=&\left\{\alpha'-\alpha:\ \alpha\in \mathcal F_\gamma, \alpha'\in\mathcal P(\vec X), (\pi_j)_\#\alpha=(\pi_j)_\#\alpha'\mbox{ for }j=1,\ldots,N\right\}.
 \end{eqnarray*}
Note that $\mathcal F_\gamma$ is composed of measures with finite-cardinality support by definition, and $\overline{\mathcal F}_\gamma$ is also composed of measures of finite support, because $\#\mathrm{spt}((\pi_j)_\#\alpha)\le \#\mathrm{spt}(\alpha)$ for all $j$ and $\#\mathrm{spt}(\alpha')\le \prod_j \#\mathrm{spt}((\pi_j)_\#\alpha)$. Then the set of ``counterexample'' costs for which $\gamma$ is cyclically monotone, but $\overline\gamma$ has better cost than $\gamma$, is given by
\[
 \mathcal C_\mathrm{cex}:=\left\{\tilde c:\vec X\to[0,+\infty]:\ \langle \tilde c, \gamma -\overline\gamma\rangle > 0,\quad \forall\ \alpha'-\alpha\in\overline{\mathcal F_\gamma}, \quad\langle\tilde c, \alpha'-\alpha\rangle\ge 0\right\}.
\]
Our assumption that a counterexample cost exists (which we showed for $\vec X=\mathbb N^N, N\ge 3$), means that $\mathcal C_\mathrm{cex}\neq\emptyset$, and we now consider the question of whether or not 
\[\mathcal C_\mathrm{cex}\cap\{\tilde c:\vec X\to [0,+\infty]: \ \forall \vec x\in\vec X,\ c(\vec x)<+\infty\}\neq\emptyset
\]
as well. Note that $\mathcal F_\gamma$ and $\overline{\mathcal F}_\gamma$ are convex, but not closed under weak-$*$ convergence of measures since $\gamma$ belongs to the closure of $\mathcal F_\gamma$ but not to $\mathcal F_\gamma$, and $-(\gamma-\overline\gamma)$ is in the closure of $\overline{\mathcal F}_\gamma$ but not in $\overline{\mathcal F}_\gamma$. As $\mathcal C_\mathrm{cex}$ is composed on functionals of the form $\langle \tilde c, \cdot\rangle$ which are required to be nonpositive on $\overline{\mathcal F}_\gamma$ and negative on $-(\gamma-\overline\gamma)$, this directly shows that all $\tilde c\in \mathcal C_\mathrm{cex}$ are not continuous under weak-$*$ convergence:
\begin{lemma}
 Let $\vec X=X_1\times\cdots\times X_N$ and $X_1,\ldots,X_N$ be Polish spaces, and assume that $c:\vec X\to (-\infty,+\infty]$ is a cost for which there exist $\gamma\in \mathcal P(\vec X)$ of finite cost, which is is cyclically monotone and not optimal. Then the assignment $\gamma\mapsto\langle c,\gamma\rangle$ defined on $\left\{\gamma\in\mathcal P(\vec X):\ \langle c,\gamma\rangle\in (-\infty,+\infty]\right\}$ is not continuous with respect to weak-$*$ convergence.
\end{lemma}
Nevertheless, we cannot exclude that there exist counterexamples $c$ which are everywhere finite:
\begin{oprob}
Do there exist Polish spaces $X_1,\ldots,X_N$ and a cost function $c:\vec X\to\mathbb R$ for which there exists a $c$-cyclically monotone non-optimal transport plan $\gamma\in\mathcal P(\vec X)$?
\end{oprob}

\section{Cyclical monotonicity and $N$-flows}
\subsection{Definition of $N$-graphs ad $N$-flows}
We give here a principle which allows to construct more complicated counterexamples with $c$ taking the value $+\infty$. The idea is that we can interpret $X^N$ as a (continuous or discrete) directed hypergraph and $c:X^N\to(-\infty,+\infty]$ as a weight on $X^N$. 

\medskip

\noindent\textbf{The case $N=2$.} The simplest case $N=2$ gives usual graphs: We then associate to $c:X^2\to(-\infty, +\infty]$ the weighted directed graph
\[
 G=(V,E,w)\quad\mbox{where}\quad \left\{\begin{array}{l}V:=X,\\ E:=\{(x,y):\, x,y\in X,\, c(x,y)<+\infty\},\\ w:E\to\mathbb R\mbox{ is given by }w(x,y):=c(x,y).\end{array}\right.
\]
Here we obtain usual (countable) weighted graphs if $X$ is countable, and a generalization thereof otherwise. We will mostly concern ourselves with atomic measures and discrete spaces $X$ for the time being.

\medskip

\noindent\textbf{The $3$-graph associated to our counterexample.} Weighted directed hypergraphs are defined analogously with directed edges replaced by $N$-ples of points in $X$. For example, the construction from the previous section gives a $3$-graph with vertex set given by three copies of $\mathbb N$ 
\[
 V:=\mathbb N\times\{1,2,3\}.
\]
If we abbreviate the $3$-edge $((k,1),(l,2),(m,3))\in V^3$ by $(k,l,m)$ then the hypergraph relevant to the previous section has $3$-edge set (see Figure \ref{fig:graph})
\[
 E:=\{(1,1,1)\}\cup\{(k,k,k+1), (k,k+1,k), (k+1,k,k):\, k\in\mathbb N\}.
\]
\begin{figure}\label{fig:graph}
 \includegraphics[height=6cm]{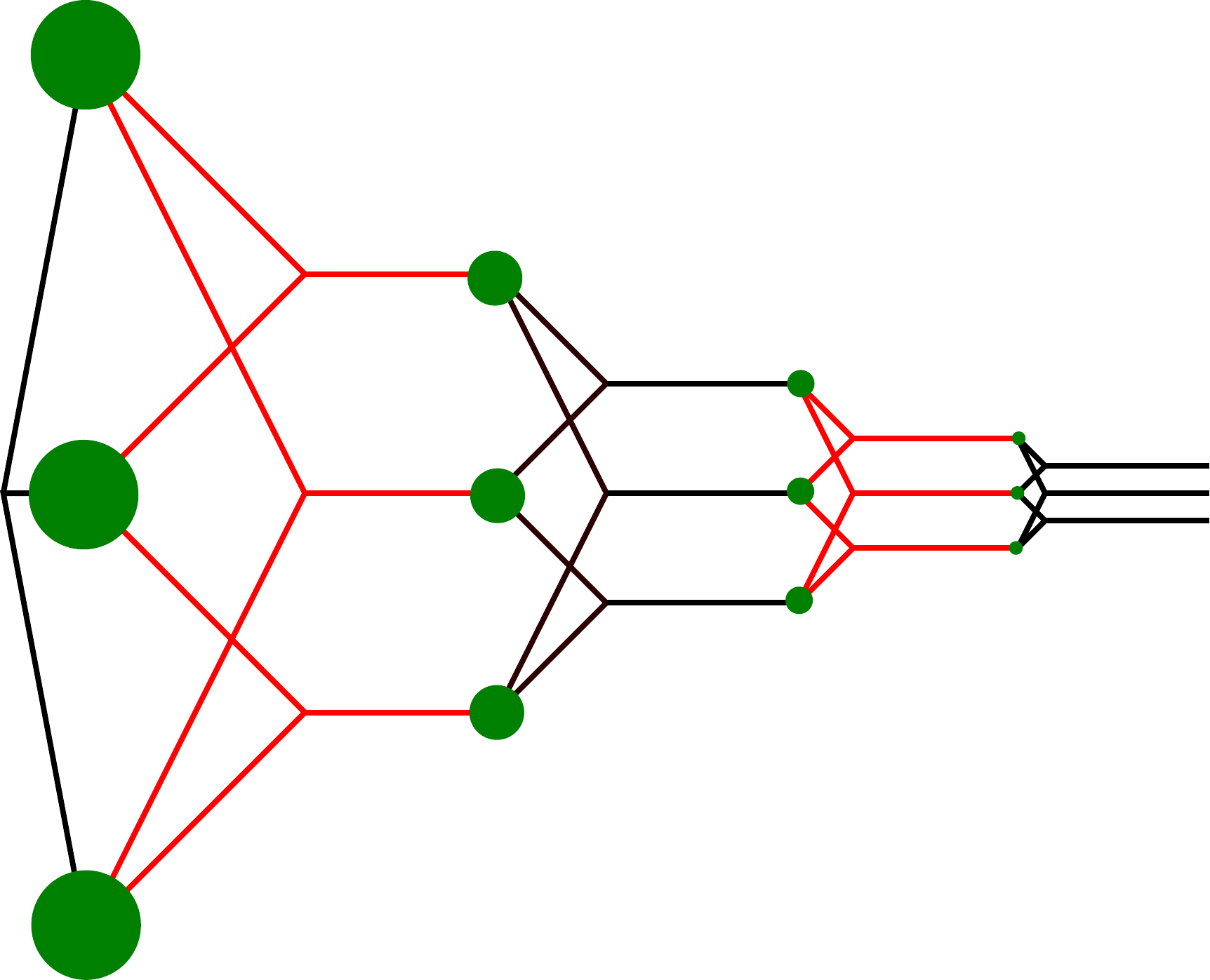}
 \caption{We depict here part of the $3$-graph associated to the cost $c$ from the counterexample we constructed in Proposition \ref{prop:counterexample1}. Here each triple $(a,b,c)$ such that $c(a,b,c)<+\infty$ (which gives an edge of the $3$-graph) is denoted by a triod, and the green points correspond to elements in $V$. The black triods correspond to the support of the transport plan $\bar\gamma$ and the red ones correspond to the support of $\gamma$. The vertical diameter of triods and circles are roughly representing the weights that corresponding points are given in $\mu, \bar\gamma$ and $\gamma$.}
\end{figure}

Anytime we study optimal transport with the cost $c$ from the previous section, the only edges which matter to our problem are the edges in $E$, due to $c$ being infinite on couplings outside $E$. This is why we can say that the $3$-graph $(V,E)$ with weight given by $c$ encodes our problem.

\medskip

\noindent\textbf{Definition of $N$-graphs and $N$-flows.} More generally, if we consider $N$-marginal optimal transport on the set $X$, we have vertex set $\tilde X_N:=X\times\{1,\ldots,N\}$ and $N$-edges set $\tilde E_N$ given by the $N$-ples $((x_1,1),\ldots,(x_N,N))$ such that $c(x_1,\ldots,x_N)\in\mathbb R$. Such edges will again be denoted $(x_1,\ldots,x_N)$, forgetting the second index, for simplicity. Again we put weight according to $c$ on $N$-edges. We call $(\tilde X_N, \tilde E_N, w)$ the \emph{$N$-graph associated} to the transport cost $c:X^N\to(-\infty,+\infty]$. Note that this graph could possibly be uncountable. See Figure \ref{fig:4graph} for an example of $4$-graph.

\begin{figure}\label{fig:4graph}
 \includegraphics[height=5cm]{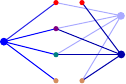}
 \caption{We show here a $4$-graph with three $4$-edges. We indicate the points belonging to $X_1,X_2,X_3,X_4$ by different colors (red, purple, green, orange), and the $4$-edges of the graph are encoded by ``spider graphs'' with $4$ legs, in different shades of blue.}
\end{figure}

To a signed weight $m:\tilde E_N\to\mathbb R$ and a subset of edges $A\subset \tilde E_N$ we can associate an \emph{$N$-flow}, which is the formal sum
\[
[A]:=\sum_{a\in A}m(a)[a].
\]
It will be useful later to also introduce the \emph{mass of the $N$-flow} $[A]$ above, defined as
\[
|[A]|:=\sum_{a\in A}|m(a)|.
\]
Given $c:\tilde E_N\to(-\infty,+\infty]$ and a finite $N$-flow, we can \emph{integrate $c$ on $[A]$} (the integral is set to be $+\infty$ if $c$ takes the value $+\infty$ on $A$):
\[
 \langle c, [A]\rangle:=\sum_{a\in A}m(a)c(a).
\]
To an $N$-flow $[A]$ as above we associate its \emph{$N$-boundary}, which is an $N$-ple of atomic measures given by 
\[
 \partial^{(N)}[A]:=\left(\mu_1,\ldots,\mu_N\right),\quad\mu_i:=\sum_{a=(a_1,\ldots,a_N)\in A}m(a)\delta_{a_i},
\]
whenever the sum is defined (in particular it is, if $A$ is finite). If $\vec A$ has $\partial^{(N)}\vec A=0$ then we say that $\vec A$ is \emph{closed} (see Figure \ref{fig:3flow} for an example).
\begin{figure}\label{fig:3flow}
 \includegraphics[height=5cm]{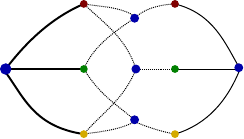}
 \caption{We show here an example of a closed $3$-flow with over a $3$-graph with 5 edges. The $3$-edges are represented by $3$-legged spider graphs similarly to Figure \ref{fig:4graph}. To visually encode the weights on the $3$-graph, we use thickness to represent the absolute value of the weights, and dashed edges correspond to negative weights while non-dashed ones correspond to positive weights. Here the thick $3$-spider has weight $1$, the dashed thin ones have weights $-1/2$ and the non-dashed thin one has weight $1/2$. Again the color blue is reserved to $3$-edges, while the other colors (dark red, green and yellow) encode to which one of $X_1,X_2,X_3$ the different endpoints of the $3$-edges belong.} 
\end{figure}

\begin{rmk}
 We could easily generalize and adapt the above definitions to the case that $m$ is a measure over $\tilde E_N$, in case the latter is infinite, replacing the sums in our definitions by suitable integrals. However, defining $c$-cyclical monotonicity does not \emph{per se} require us to talk about measures, as the definition \eqref{cycmon} is be formulated in purely combinatorial terms. Indeed, it turns out that the case in which the support $A$ of $[A]$ is finite is all we need, and this allows us to avoid discussing measurability issues which would make the discussion superfluously complicated.
\end{rmk}
\subsection{Reformulating $c$-cyclical monotonicity in terms of $N$-flows} 
We are now ready for a combinatorial reformulation of the monotonicity condition \eqref{cycmon}, which allows us to describe more complicated counterexamples. It says that a set is $c$-monotone if and only if its flows are cost-minimizers at fixed boundary.
\begin{prop}\label{prop:nflow_charact}
A set $A\subset \vec X$ is $c$-monotone if and only if whenever $F\subset A$ is finite, $[F]$ and $[F']$ are finite-mass $N$-flows in the $N$-graph associated to $c$ such that $[F]$ is supported on $F$ and $\partial^{(N)}([F]-[F'])=0$, then there also holds $\langle c,[F]-[F']\rangle\le 0$.
\end{prop}
\begin{proof}
The proof is an easy verification, based on the following fact: an equivalent condition for $A$ to be $c$-monotone is that for any measure $\alpha$ supported on a finite subset $F\subset A$ and any $\alpha'$ with marginals equal to those of $\alpha$, there holds $\langle c,\alpha-\alpha'\rangle\le 0$. The measure $\alpha$ support included in the finite set $F$ implies that $\alpha=\sum_{f\in F}\alpha(\{f\})\delta_f$. We can thus associate the $N$-flow $[F]:=\sum_{f\in F}\alpha(\{f\})[f]$ to $\alpha$. Similarly, since $\alpha'$ has atomic marginals it is itself atomic too and we can associate another $N$-flow $[F']$ to $\alpha'$ as above. The condition of $\alpha'=\alpha$ is equivalent to requiring $\partial^{(N)}([F]-[F'])=0$ as can be easily verified, and the condition $\langle c,\alpha-\alpha'\rangle\le 0$ is equivalent to $\langle c,[F]-[F']\rangle\le 0$.
\end{proof}
The counterexample from the previous section is generalized by the following:
\begin{cor}\label{cor:acyclic}
 If the $N$-graph associated to $c$ supports no finite closed $N$-flows then all transport plans of finite $c$-cost are $c$-montone.
\end{cor}
\begin{proof}
Let $\gamma$ be a transport plan with $\gamma(\{c<+\infty\})=1$. We claim that the condition $\partial^{(N)}([F]-[F'])=0$ is never verified for $[F]$ of finite support contained in the set $\{c<+\infty\}$. Indeed, if $[F]$ has finite support then so does $\partial^{(N)}[F]=-\partial^{(N)}[F']$, and then $[F']$ also has finite support, and thus $[F]-[F']$ would be a finite closed $N$-flow supported on $\{c<+\infty\}$, which is precluded by our hypothesis. The condition of Proposition \ref{prop:nflow_charact} is then automatically verified.
\end{proof}
\subsection{Special properties available for the case $N=2$}
Before discussing some necessary and some sufficient conditions for the existence of $N$-flows, it is instructive to consider the case $N=2$, reinterpreting the proof methods \cite{pratelli} and \cite{beiglbock}. 

\medskip

We define a \emph{$2$-flow-path} to be any $2$-flow of the form 
\[
 [A]=\sum_{i=0}^N\left(\left[\left(a_1^{(j)},a_2^{(j)}\right)\right] - \left[\left(a_1^{(j+1)},a_2^{(j)}\right)\right]\right)\quad\mbox{or}\quad[A]=\sum_{i=0}^N\left(\left[\left(a_1^{(j)},a_2^{(j)}\right)\right] - \left[\left(a_1^{(j)},a_2^{(j+1)}\right)\right]\right),
\]
where $m\in\mathbb R$, $N\in\mathbb N$ and $a_1^{(j)}\in X_1$ and $a_2^{(j)}\in X_2$ for $j=0,\ldots,N$. If $a_1^{(N+1)}=a_1^{(0)}$ or respectively $a_2^{(j+1)}=a_2^{(0)}$ in the above but no other pair of points coincide, then we call $[A]$ a \emph{$2$-flow-loop}.

\medskip

For the next definition, let $[A],[A_1],[A_2]$ be three $2$-flows such that $[A]=[A_1]+[A_2]$. In general we have $|[A]|\le|[A_1]|+|[A_2]|$. If the equality sign holds in the latter inequality, we say that $[A_1],[A_2]$ form a \emph{decomposition of $[A]$ wihtout cancellations}.

\medskip

It is easy to see that if a closed $2$-flow has the same support as $[A]$ above, then automatically it is a nonzero multiple of $[A]$. The principle that ``mass is conserved along $2$-flow-loops'' (which lacks an easy analogue for $N$-flows if $N>2$, as a consequence of the counterexample in the proof of proposition \ref{prop:counterexample1}) helps us to prove the following:
\begin{lemma}\label{lem:fin2flow}
 A finite $2$-flow $[A]$ closed if and only if $[A]$ is a finite sum of weighted $2$-flow-loops without cancellations.
\end{lemma}
\begin{proof}
 The ``if'' part of the implication is obvious, as a sum of closed $2$-flows is closed. The proof of the other implication is an adaptation of a folklore result on weighted directed graphs satisfying the Kirchhoff balance equation at each vertex, so we only sketch it for the convenience of the reader. The basic idea is to start with an edge of the $2$-graph which is assigned lowest weight in absolute value, and use the Kirchhoff law to find a path of edges which is locally satisfying Kirchhoff law. When such path first self-intersects, we find one $2$-flow-loop $[A']$ such that $[A],[A]-[A']$ forms a decomposition of $[A]$ without cancellations. By repeating the procedure with $[A]$ replaced by $[A]-[A']$ we can keep diminishing the mass of the part not yet decomposed into loops, while diminishing the total number of loops in the graph corresponding to $A$. When the remaining graph is one single loop, the procedure ends.
 
 \medskip
 
 In order to describe how we single out one loop in the above sketch of proof, consider the nontrivial case $[A]\neq 0$ and assume (as we can do up to change of sign symmetry) that for some $(a_1,a_2)\in A$ there holds $m:=m(a_1,a_2)=\min\{|m(a'_1,a'_2)|:\ (a'_1,a'_2)\in A\}$. We start building a ``$2$-path-flow'' by defining $\gamma_1:=m [(a_1,a_2)]$. By considering the $\delta_{a_2}$-coefficient in the expression for $\partial^{(2)}[A]=0$, and due to the fact that $m=m(a_1,a_2)>0$ is minimal by definition, we find that there exists $a_1'\neq a_1\in X_1$ such that $(a'_1,a_2)$ belongs to the support of $[A]$ and $m(a'_1,a_2)\le - m$. We add this edge with coefficient $-m$ to $\gamma_1$ and obtain $\gamma_2:=m\left([(a_1,a_2)]-[(a'_1,a_2)]\right)$. By focusing now on the $\delta_{a'_1}$-coefficient in the expression for $\partial^{(2)}[A]=0$ we find $a'_2\in X_2$ such that $(a'_1,a'_2)\in A$ and this edge has weight $m(a'_1,a'_2)\ge m$, and we define $\gamma_3:=m\left([(a_1,a_2)]-[(a'_1,a_2)]+[(a'_1,a'_2)]\right)$. The procedure continues this way till we visit one point which was already visited before, and in this case a subpath of the $2$-flow-path we reached is a $2$-flow-loop. The procedure ends before step $k\le[A]/m$ because $[A]$ is of finite mass, we have that $\gamma_k, [A]-\gamma_k$ is a decomposition of $[A]$ without cancellations for every $k$, and $|\gamma_k|=km$ for every $k$.
\end{proof}
The above proof strategy also gives the following countable version of Lemma \ref{lem:fin2flow}:
\begin{lemma}\label{lem:finmass2flow}
 A finite-mass $2$-flow over a countable set $X$ is a linear combination of $2$-flow-loops without cancellations.
\end{lemma}
If $X$ is countable, $N=2$, and $c$ is arbitrary, Lemma \ref{lem:finmass2flow} can be applied to optimal transport as follows. Given two distinct plans $\gamma,\bar\gamma$ with finite cost and with the same marginal, they define flows $[G_\gamma]$ and $[G_{\bar\gamma}]$ such that $[G_\gamma]-[G_{\bar\gamma}]$ is a closed $2$-flow. So it is a superposition of finite loops. If we use $c$-monotonicity assumption on $\gamma$ for each such loop, we find by linearity that $\langle c,[G_\gamma]\rangle\le \langle c,[G_{\bar\gamma}]\rangle$, which due to the arbitrarity of $\bar\gamma$ implies the optimality of $\gamma$. This is basically the strategy of \cite{pratelli}. The failure of the above lemmas for $N\ge 3$ shows that such strategy does not extend for general $c$ due to the higher combinatorial complexity of this case.

\section{Generalized Smirnov decomposition for $N$-flows}
In Proposition \ref{prop:nflow_charact} we translate $c$-monotonicity in terms of a condition on closed finite $N$-flows supported in the $N$-graph associated to $c$, and Corollary \ref{cor:acyclic} produces counterexamples using the absence of such $N$-flows. It becomes useful to consider the classification of $N$-flows depending on the type of cycles they contain. Here we concentrate more on such properties for the case of $N$-graphs. Note that as before, we could consider countable $N$-graphs and avoid measurability issues, or consider more general $N$-graphs, on which weights and masses would be well-defined by introducing a measure space structure. 

\medskip 

The properties emerging in this context are analogous to those from the decomposition of currents/flows into cycles and acyclic parts from \cite{smirnov} and \cite{paoste}, from which we imitate the terminology. We also mention that the use of Smirnov decomposition in Optimal Transport type problems is not new, and has appeared in another setting in \cite{petrachebrasco}.

\medskip 

We start with the following definitions:
\begin{itemize}
\item An $N$-flow is called a \emph{cycle} if it is closed.
\item An $N$-flow $[A]$ is a \emph{subflow} of the $N$-flow $[B]$ if there exists an $N$-flow $[C]$ such that 
\[[A]+[C]=[B]\quad\mbox{and}\quad|[A]|+|[C]|=|[B]|.\] 
\item If $[A]=0$ or $[A]=[B]$ then $[A]$ is a \emph{trivial subflow} of $[B]$. 
\item An $N$-flow is called \emph{acyclic} if it has no nonzero cyclic subflows.
\item An $N$-flow is called \emph{solenoidal} if it is cyclic and has no finite cyclic subflows.
\end{itemize}
We have the following basic decomposition result, similar to \cite{smirnov} and \cite{paoste}. A similar result holds with measurability assumptions for more general $[A]$, corresponding to the case where transport plans correspond to general measures. We do not treat this generalization here, as the below countable case already contains the basic principle. As the method for the case of currents in the above papers is very similar, we only sketch the proof here.
\begin{prop}\label{prop:acycdec}
 If $[A]$ is an $N$-flow with countable support such that $|[A]|<\infty$ then there exist unique subflows $[A_1],[A_2], [A_3]$ of $[A]$ such that $[A_1]$ is acyclic, $[A_2]$ is solenoidal and $[A_3]$ is a superposition of finite cyclic $N$-flows, such that $[A]=[A_1]+[A_2]+[A_3]$.
\end{prop}
\begin{proof}[Sketch of proof of Proposition \ref{prop:acycdec}]
We call a sequence of $N$-flows $[B_1],[B_2],\cdots$ an \emph{increasing sequence} if each $[B_k]$ has finite mass and is a subflow of $[B_{k+1}]$. We will use the fact that an increasing sequence of subflows has a limit which is itself a subflow. If $[A]$ has a finite-support cyclic subflow, we may remove the largest-mass finite cycle subflow from $[A]$ and diminish its mass. This can be done an at most countable number of times, as each time the mass of $[A]$ diminishes by a nonzero amount. Thus we may assume that $[A]$ has no finite cycle subflow and that $[A_3]=0$. Now note that when we have $[A]=[A']+[A'']$ and $|[A]|=|[A']|+|[A'']|$, there also holds $\partial^{(N)}[A]=\partial^{(N)}[A']+\partial^{(N)}[A'']$ and $|\partial^{(N)}[A]|=|\partial^{(N)}[A']|+|\partial^{(N)}[A'']|$. Thus we may remove from $A$ acyclic subflows in mass-decreasing order, and we terminate the procedure in a countable number of steps, as each such flow removes a positive amount of boundary mass. The sum $[A_2]$ of all such subflows is then acyclic and is a subflow of $[A]$ and the difference $[A]-[A_2]=:[A_1]$ is then by construction solenoidal.
\end{proof}
In order to make Corollary \ref{cor:acyclic} more concretely useful, we would need to consider the following.
\begin{oprob} Find necessary conditions and sufficient conditions for an $N$-graph $G$ to support no finite cyclic $N$-flows. Equivalently, find necessary conditions and sufficient conditions such that $G$ supports only acyclic and solenoidal flows.
\end{oprob}

\end{document}